\documentclass[11pt]{article}

\usepackage{amsmath,amsthm,mathrsfs}
\usepackage{amssymb,latexsym}

\usepackage[mathscr]{euscript}

\usepackage{enumerate}

\usepackage{xcolor}

\newcommand{\BB}{{\mathcal B}}

\newcommand{\EE}{{\mathcal E}}
\newcommand{\FF}{{\mathcal F}}

\newcommand{\BE}{{\mathbb E}}
\newcommand{\BF}{{\mathbb F}}

\newcommand{\BM}{{\mathbb M}}

\newcommand{\BP}{{\mathbb P}}
\newcommand{\BR}{{\mathbb R}}

\newcommand{\DDr}{{\mathscr D}}

\newcommand{\sgn}{\mbox{\rm sgn}}
\newcommand{\fch}{{\mathbf{1}}}

\newtheorem{theorem}{\bf Theorem}[section]
\newtheorem{proposition}[theorem]{\bf Proposition}
\newtheorem{lemma}[theorem]{\bf Lemma}
\newtheorem{corollary}[theorem]{\bf Corollary}

\theoremstyle{definition}
\newtheorem{definition}[theorem]{Definition}

\newtheorem{example}[theorem]{\bf Example}

\newtheorem{remark}[theorem]{Remark}

\numberwithin{equation}{section}

\begin{document}

\title {Nonlinear  Hardy--Stein type  identities for harmonic functions relative to symmetric integro-differential operators}
\author {Tomasz Klimsiak and Andrzej Rozkosz
\\
{\small Faculty of Mathematics and Computer Science,
Nicolaus Copernicus University} \\
{\small  Chopina 12/18, 87--100 Toru\'n, Poland}}
\date{}
\maketitle
\begin{abstract}
We show identities of Hardy--Stein type  for harmonic functions relative to integro-differential  operators corresponding to general symmetric regular Dirichlet forms satisfying the absolute continuity condition. The novelty is that we consider operators of mixed type containing both local and nonlocal component. Moreover, the identities are proved for  compositions of harmonic functions and general convex functions. We also provide some  conditional identities, i.e.  identities for ratios of harmonic functions. As an application we give a characterization of norms in harmonic Hardy spaces and prove  Littlewood--Paley type estimates for square functions. To illustrate general results, we  discuss in some details the case of divergence  form operator and  purely nonlocal operator defined by some jump kernel. Our proofs are rather short and use mainly probabilistic methods.
\end{abstract}
{\small \noindent{\bf Mathematics Subject Classification (2020).} Primary 31C05, 31C25; Secondary 60J45.}

\footnote{{\em E-mail addresses}: {\tt tomas@mat.umk.pl} (T. Klimsiak), {\tt rozkosz@mat.umk.pl} (A. Rozkosz)}

\section{Introduction}
\label{sec1}

Let $E$ be a locally compact separable metric space and $m$ be a Radon measure on $E$ with full support. We consider a symmetric regular Dirichlet form
$(\EE,\DDr(\EE))$ on $L^2(E;m)$.
It is well known that it  admits the unique 
Beurling--Deny decomposition
\begin{equation}
\label{eq1.1}
\EE(u,v)=\EE^{(c)}(u,v)+\EE^{(j)}(u,v)+\int_Euv\,d\kappa,
\quad u,v\in \DDr(\EE)\cap C_c(E),
\end{equation}
where  $\EE^{(c)}$ is a symmetric form on  $\DDr(\EE) \cap C_c(E)$
satisfying  the strong local property, 
$J$ is a symmetric positive Radon measure (called jumping measure) on 
$E\times E\setminus \frak d$, 
where $\frak d=\{(x,x): x\in E\}$, and $\kappa$ (called killing measure) is a positive Radon measure on $E$. 

In the present paper, we are interested in Hardy--Stein type identities for functions of the form $\varphi(u)$, where  $\varphi:\BR\to\BR_+$ is  a convex function
and $u$ is a general harmonic function relative to $\EE$ in some open subset $D$ of $E$.
We  denote by  $\varphi'_{-}$ the left derivative of $\varphi$ and by  $\mu_{\varphi}$ the  measure being the distributional derivative of $\varphi'_{-}$. An  important subclass of convex functions we consider consists of $\varphi$ such that   
\begin{equation}
\label{eq1.2}
\mu_\varphi(dx)=g(x)\,dx\quad\mbox{for some }\quad g\in\mathcal B^+(\BR).	
\end{equation}
We define harmonic functions probabilistically as functions having some mean value property.
In many interesting cases, however,  they can be described analytically as functions $u$  such that
\begin{equation}
	\label{eq1.8}
	\EE(u, v)=0 \quad \text { for every } v \in C_c(D) \cap \mathscr D(\EE).
\end{equation}
Model examples  of convex functions are
\[
\varphi(x)=|x|,\qquad \mu_{\varphi}(dx)=2\delta_0,
\]
where $\delta_0$ is the Dirac measure concentrated at point $0$, and
\begin{equation}
\label{eq1.3}
\varphi(x)=|x|^{p},\qquad
g(x)=p(p-1)\mbox{\rm sgn}(x)\fch_{\{|x|>0\}} |x|^{p-2},\quad p>1,
\end{equation}
where $\mbox{sgn}(x)=\mathbf1_{(0,\infty)}(x)-\mathbf1_{(-\infty,0]}(x)$, i.e. $\sgn$ is the left derivative of $x\mapsto|x|$.

Assume that $\EE$ satisfies the absolute continuity condition (see Section \ref{sec2}) and  for any relatively compact open set $U$ such that $\bar U\subset D$ 
(written $U\subset\subset D$) the part of  $\EE$ on $U$ is transient.
The aim of the paper is to show that under these additional assumptions,   if $u$ is a  harmonic function in $D$ 
and $\varphi$ is a convex function satisfying (\ref{eq1.2}), then for any 
$U\subset\subset  D$ and  $x\in U$,
\begin{align}
\label{eq1.4}
\int_{U^c}\varphi(u(z))\,P_U(x,dz)&=\varphi(u(x))+\frac12
	\int_UG_U(x,z)g(u(z))\,\mu^c_{\langle u\rangle}(dz) \nonumber\\
	&\quad+2\int_UG_U(x,z)\int_{E}F_\varphi(u(z),u(y))\,J(dz,dy)\nonumber\\
	&\quad+\int_UG_U(x,z)\int_{E}F_{\varphi}(u(z),0)\,\kappa(dz).
\end{align}
In (\ref{eq1.4}), $P_U(x,dz)$ is the Poisson kernel for the set $U$ (relative to $\EE$),
$G_U$ is the Green function for $U$  (relative to $\EE$), $\mu_{\langle u\rangle}^c$
is the local part of the energy measure of $u$, i.e. a  unique positive Radon measure on $E$ which for bounded $u$ is defined by the equation 
\[
\int_E f\,d\mu_{\langle u\rangle}^c=2\mathcal{E}^{(c)}(u, uf)-\mathcal{E}^{(c)}(u^2, f),
\quad f \in \mathscr{D}(\EE) \cap C_c(U). 
\]
Finally, $F_\varphi:\BR\times\BR\rightarrow\BR$ is the so-called Bregman divergence defined  by
\[
F_\varphi(a,b)=\varphi(b)-\varphi(a)-\varphi'_{-}(a)(b-a)
\quad a,b\in\BR.
\]
In the general case (\ref{eq1.4}) holds with the term involving $g$  replaced by 
\begin{equation}
\label{eq1.7}
\frac12\int_{\BR}\int_UG_U(x,z)\,l_a(dz)\,\mu_\varphi(da),
\end{equation}
where $\{l_a,a\in\BR\}$  is a family of smooth measures on $E$ uniquely determined by $\EE$ and  $u$  (we describe it more precisely in Section \ref{sec3}). 

Results of type (\ref{eq1.4}) were proved in \cite{BDL,BGPPR2} for some special forms and $\varphi(x)=|x|^p$ with $p>1$. In this case
\[
F_{\varphi}(a,b)=F_p(a,b):=|b|^p-|a|^p-p\,\sgn(a)|a|^{p-1}(b-a),
\quad  F_p(a,0)=(p-1)|a|^p.
\] 
In the course of writing this  paper we learned about the paper  \cite{BKPP}, where the  formula (\ref{eq1.4}) is proved for the form corresponding to a L\'evy operator and some class of convex functions.

Let $L$ denote the generator of $\EE$, that is $(L,\DDr(L))$ is the unique nonpositive definite self-adjoint operator on $L^2(E;m)$ such that 
\begin{equation}
\label{eq1.5}
\DDr(L)\subset\DDr(\EE),\qquad \EE(u,v)=(-Lu,v),\quad u\in\DDr(L),\,v\in\DDr(\EE).
\end{equation}
In \cite{BDL} separately two cases are considered: $L=\Delta$ (Laplace operator, i.e. $\EE^{(j)}=0$, $\kappa=0$) and $L=\Delta^{\alpha/2}$ with $\alpha\in(0,2)$ (fractional Laplace operator, i.e. $\EE^{(c)}=0$, $\kappa=0$). In \cite{BGPPR2} (see also \cite{BGPPR1}) the results of \cite{BDL} for fractional Laplace operator are generalized to some pure jump forms with no killing (see Section \ref{sec5} for more details). The paper \cite{BKPP} essentially covers the case where $\kappa=0$ and  $\EE^{(c)}+\EE^{(j)}$ corresponds to a symmetric L\'evy operator whose model example is $\Delta-(-\Delta)^{\alpha/2}$. As for $\varphi$, in \cite{BKPP} it is assumed that it is a moderate Young function with absolutely continuous derivative.

Our results are much more general. For instance, they cover the case $p=1$ as well as non-moderate Young functions $\varphi$, and large class of integro-differential operators $L$.  
In the paper, we  also prove identities of type (\ref{eq1.4}) for functions of the form $\varphi(u/h)$, where  $u,h$ are singular harmonic in $D$ (i.e. $u=h=0$ on $D^c$), $h>0$ and   
$\varphi$ satisfies  (\ref{eq1.2}). We call them conditional Hardy--Stein identities. They  generalize considerably the corresponding results proved earlier in \cite{BDL}. 

Our proofs are probabilistic in nature. 
We use in an essential way the It\^o--Meyer  formula, the concept of L\'evy system for $\BM$ and some general results from the probabilistic potential theory (notably the Fukushima decomposition, Revuz duality, probabilistic interpretation of Beurling-Deny formulae  and fine topology).

The main new feature of the paper can be described concisely  as follows.
\begin{itemize}
\item Our results are proved for general harmonic functions (not only for a subclass of such functions satisfying (\ref{eq1.8})) and for general convex functions.
	
\item We consider quite general regular forms  $(\EE,\mathscr D(\EE))$. Our additional assumptions on $(\EE,\mathscr D(\EE))$
are minimal in order to guarantee the existence of the Green function $G_U$ for any $U\subset\subset D$.

\item We provide  a unified method for proving  Hardy--Stein type  identities  and their conditional versions in the very general framework.  Nevertheless, our proofs are  rather short.
\end{itemize}

As an application of Hardy-Stein identities, we state characterizations of norms in harmonic Hardy spaces 
and give Littlewood--Paley type estimate for square functions and their conditional versions. 

In the last section of our paper, we give examples of forms satisfying our assumptions  and provide additional  comments on special cases.

\section{Preliminaries}
\label{sec2}

We denote by $\BB(\BR)$ (resp. $\BB^+(\BR)$) the set of all Borel (resp. nonnegative Borel) functions on  $\BR$. For a convex function $\varphi:\BR\rightarrow\BR_+:=[0,\infty)$ we denote by $\varphi'_{-}$ its left derivative, i.e. 
$\varphi'_{-}(x)=\lim_{t\to 0^-}(\varphi(x+t)-\varphi(x))/t$, $x\in\BR$.
We say that functions $u$, $v$ are comparable on a set $B$, written  $u \asymp v$ on $B$, 
if there exist constants $c, C>0$ such that
$$
c v(x) \leq u(x) \leq C v(x),\quad x\in B.
$$

\subsection{Dirichlet forms and associated Hunt processes}

In the paper, $E$ is a locally compact separable metric space and $D$ is an open subset of $E$.
Let $\partial$ be a one-point compactification of 
$E$ if $E$ is non-compact, and an isolated point if $E$ is compact. We adopt the convention 
that every function $f$ on $E$ is automatically extended to $E_\partial=E\cup\{\partial\}$ by letting $f(\partial)=0$.

Let $m$ be a  Radon measure on $E$ with full support. We denote by $(\EE,\mathscr D(\EE))$ a regular symmetric Dirichlet 
form  on $L^2(E;m)$.  We always assume that for any open set  $U\subset\subset D$ the part of the form $\EE$ on $U$ is transient.
Quasi-notions used in the paper (capacity, quasi continuity, quasi open sets etc.)  associated with $\EE$ are defined as in \cite[Section 2.1]{FOT}.  

We denote by  $\BM=(\Omega,\BF=(\FF_t)_{t\ge0},X=(X_t)_{t\in[0,\infty]}, (\BP_x)_{x\in E_\partial})$  an $m$-symmetric Hunt process on $E$ associated with $(\EE,\mathscr D(\EE))$.
Its life time is denoted by $\zeta$ and the associated translation operators by $\theta_t$, $t\ge0$. In the paper, we assume that $\BM$ satisfies  the  absolute continuity condition, i.e. 
\[
\int_0^\infty e^{-t}p_t(x,dy)\,dt\ll m(dy)\quad\mbox{for every  $x\in E$},
\]
where  $p_t$ is the transition function of $\BM$.
For a nearly Borel   $U\subset E$ we set
\[
\tau_U=\inf\{t>0:X_t\notin U\}.
\]
We define the potential  operator $R^U$ associated with the part of the form $\EE$ on $U$ by
\[
R^Uf(x)=\BE_x\int^{\tau_U}_0f(X_t)\,dt,\quad x\in E, \,f\in\BB^+(E),
\]
where $\BE_x$ is the expectation with respect to $\BP_x$. Since $\BM$ satisfies the absolute continuity condition and $\mathbb M^U$ is transient
for any open $U\subset\subset D$, from
\cite[Lemma 4.2.4]{FOT} and an approximation argument it follows that there is a nonnegative jointly measurable  function $G_U(x,y)$, $x,y\in E$ (called Green function for  $U$ relative to $\EE$) such that
\[
R^Uf(x)=\int_{E}G_U(x,y)f(y)\,m(dy),\quad x\in E, \,f\in\BB^+(E).
\]
For a nonnegative  Borel measure $\mu$ on $U$ we further set
\[
R^U\mu(x)= \int_UG_U(x,y)\,\mu(dy),\quad x\in U.
\]
We also set
\[
P_U(x,B)=P_x(X_{\tau_U}\in B),\quad B\in\BB(E).
\]
The family of measures $\{P_U(x,\cdot),x\in E\}$ is called the Poisson kernel for $U$ (relative to $\EE$).

\subsection{Harmonic functions}

As a basic definition of harmonicity we adopt the following.

\begin{definition}
A Borel measurable function $u$ on $E$ is called harmonic in $D$ (relative to $\EE$)
if for all open $U\subset\subset D$ and $x\in U$, $\mathbb E_x |u(X_{\tau_U})|<\infty$ and $\BE_xu(X_{\tau_U})=u(x)$.
\end{definition}

We denote by  $\mathscr H(D)$  the set of all harmonic functions in $D$ relative to $\EE$. Since $\EE$ is  fixed in the whole paper except for the last section including examples, we suppress the explicit dependence of this space on $\EE$ in the notation. 

\begin{remark}
\label{rem2.2}
According to the above definition, harmonic functions in $D\subset E$ must be defined on the whole $E$. 
This reflects the nature of the notion of harmonicity in the general framework considered here, which includes non-local or mixed type operators.  However, if  $\EE$ is local, i.e. $\EE^{(j)}=0$ in \eqref{eq1.1}, then 
 the behavior of a function $u:E\to \mathbb R$  outside $D$ does not affect whether  $u$ is harmonic in $D$.
This is  because $\EE^{(j)}=0$ if and only if $\mathbb M$ is continuous (see \cite[Theorem 4.5.1]{FOT}), which in turn implies that
$\mathbb P_x(X_{\tau_U}\in D\cup\{\partial\})=1$ for any open $U\subset\subset D$ and $x\in U$.
\end{remark}

\begin{proposition}
\label{prop2.2}
A  Borel measurable function $u$ on  $E$  belongs to $\mathscr H(D)$   if and only if   for every  $U\subset\subset D$
the process $t\mapsto u(X_{t\wedge\tau_U})$ is a uniformly integrable $\mathbb P_x$-martingale for any  $x\in U$.
\end{proposition}
\begin{proof}
The sufficiency is obvious. To prove the necessity, fix an open $U\subset\subset D$.
Since $u$ is assumed to be in $\mathscr H(D)$, $u(x)=\mathbb E_x u(X_{\tau_U}),\, x\in U$.
Therefore, by the Markov property, for every  $x\in U$,
\begin{align*}
\mathbf1_{\{t<\tau_U\}} u(X_t)&=\mathbf1_{\{t<\tau_U\}}  \mathbb E_{X_t} u(X_{\tau_U})=
\mathbf1_{\{t<\tau_U\}}  \mathbb E_{x} (u(X_{\tau_U}\circ \theta_t)|\FF_t)\\&=
\mathbb E_{x} (\mathbf1_{\{t<\tau_U\}}  u(X_{\tau_U}\circ \theta_t)|\FF_t)
=
\mathbb E_{x} (\mathbf1_{\{t<\tau_U\}}  u(X_{\tau_U})|\FF_t)\\&=
\mathbf1_{\{t<\tau_U\}}  \mathbb E_{x} (u(X_{\tau_U})|\FF_t)
=\mathbf1_{\{t<\tau_U\}}  \mathbb E_{x} (u(X_{\tau_U})|\FF_{t\wedge\tau_U}),
\end{align*}
$\mathbb P_x$-a.s. Hence, for every $x\in U$,
\begin{align*}
u(X_{t\wedge\tau_D})&=\mathbf1_{\{t<\tau_U\}}  u(X_t)+
\mathbf1_{\{t\ge \tau_U\}} u(X_{\tau_U\wedge t})\\&=
\mathbf1_{\{t<\tau_U\}}  \mathbb E_{x} (u(X_{\tau_U})|\FF_{\tau_U\wedge t})
+\mathbf1_{\{t\ge \tau_U\}} \mathbb E_{x} (u(X_{\tau_U\wedge t})|\FF_{t\wedge\tau_U})\\
&=\mathbb E_{x} (u(X_{\tau_U})|\FF_{t\wedge\tau_U}),
\end{align*}
$\mathbb P_x$-a.s., which completes the proof of the necessity.
\end{proof}

Define
\[
\DDr(\EE^D)=\{u\in\DDr(\EE): u=0\mbox{ q.e. on }E\setminus D\}.
\]
We shall say that $u\in{\mathscr D}_{loc}(\EE^D)$ if for every open $U\subset\subset D$ there is $v\in\mathscr D(\EE^D)$
such that $u=v$ $m$-a.e. on $U$. Every function  $u\in\mathscr D_{loc}(\EE^D)$ admits an $m$-version that is
quasi continuous on $D$. In the paper, we assume that  every function in $\mathscr D_{loc}(\EE^D)$ when restricted to $D$ is represented by its quasi continuous version.
We say that $u\in\dot{\mathscr D}_{loc}(\EE^D)$ if there exists a sequence $(V_n)_{n\ge1}$
of finely open sets such that its union is q.e. equal to $D$
and for every $n\ge 1$ there exists  $v_n\in\mathscr D(\EE^D)$
such that $u=v_n$ q.e.  on $V_n$.

\begin{theorem}
\label{th2.3}
Let $u:E\rightarrow\BR$ be a Borel measurable  function.
\begin{enumerate}[\rm(i)]
\item If $u\in\mathscr H(D)$, then $u\in\dot{\mathscr D}_{loc}(\EE^D)$.
\item If $u\in\mathscr H(D)$ is   nonnegative  and locally bounded in $D$, then $u\in\mathscr D_{loc}(\EE^D)$.
\item If $u\in\mathscr H(D)$ is nonnegative in $D$, then  $u\wedge k\in \mathscr D_{loc}(\EE^D)$ for every $k\ge 0$.
\end{enumerate}
\end{theorem}
\begin{proof}
To show (i), consider an open set  $V\subset\subset D$ and put  
$h^1_V(x)=\BE_xu^+(X_{\tau_V})$, $h^2_V(x)=\BE_xu^-(X_{\tau_V})$. 
Both $h_1$ and $h_2$ are $(P^V_t)$-excessive, hence finely continuous on $V$.
Applying \cite[Theorem 5.3]{KR}  to $h^i_V$, and $V^i_k=\{h^i_V<k\}$, $i=1,2$, gives (i) by arbitrariness of $V\subset\subset D$
and $k$. Part (ii) follows directly from  \cite[Theorem 5.3]{KR}.
We now turn to the proof of (iii). 
By the definition of a harmonic function in $D$ one easily concludes that $w_k:=u\wedge k$ is $(P^D_t)$-excessive.
Let $U\subset\subset D$ be an open set. By \cite[Proposition III.1.5]{MR}, there exists a function  $h^k_U\in \mathscr D(\EE^D)$ such that $h^k_U=w_k$ $m$-a.e. on $U$. This establishes (iii). 
\end{proof}

\begin{definition}[Energy measure]
For  bounded $v\in \mathscr D(\EE)$ we denote by $\mu^c_{\langle v\rangle}$  the unique smooth measure  such that 
\[
\int_Ef\,d\mu^c_{\langle v\rangle}
=2\mathcal{E}^{(c)}(v,f)-\mathcal{E}^{(c)}(v^2,f), \quad f \in \mathscr{D}(\EE) \cap C_c(E),
\]
and then we extend the definition  to $v\in \mathscr D(\EE)$ as the  limit of $\mu^c_{T_k(v)}$, where  $T_k(s)=\min(\max(-k,s),k)$ (see the comments in   \cite[ p. 126]{FOT}). 
\end{definition}

\begin{remark}
By \cite[Exercise 4.3.12]{CF}, $\mu^c_{\langle v\rangle}$ is the unique smooth measure 
that is in the Revuz correspondence with the positive continuous additive functional  $[M^{[v]}]^c$ of $\BM$, where $M^{[v]}$ is the martingale additive functional  of $\BM$ of finite energy  from the Fukushima decomposition of the additive functional $v(X)-v(X_0)$. As a result, $\mu^c_{T_k(v)}$
is  increasing in $k$. Furthermore, for  $v\in \mathscr D(\EE)$ we may define $\mu^c_{\langle v\rangle}$ equivalently as  the unique smooth measure  such that for every $k>0$,
\[
\int_Ef\mathbf1_{\{-k<v\le k\}}\,d\mu^c_{\langle v\rangle}=2\mathcal{E}^{(c)}(T_k(v), T_k(v) f)-\mathcal{E}^{(c)}((T_k(v))^2, f),
\quad f \in \mathscr{D}(\EE) \cap C_c(E).
\]
\end{remark}

\begin{remark}
\label{rem2.5}
 (i) By \cite[Theorem  4.3.10]{CF},  for any $v\in\dot{\DDr}_{loc}(\EE^D)$, there exists a unique  smooth measure $\mu^c_{\langle v\rangle}$ on $D$ such that 
for any nearly Borel quasi open set  $V\subset\subset  D$ and 
any bounded function $\eta_V\in \mathscr D(\EE)$  satisfying $\eta_V=v$ q.e. on $V$
\begin{equation}
\label{eq2.8}
(\mu^c_{\langle v\rangle})_{\lfloor V}=(\mu^c_{\langle \eta_V\rangle})_{\lfloor V}.
\end{equation}
(ii) If we additionally know  that there exists a nearly Borel quasi open set $U\subset\subset D$
such that $Y_t= v(X_{t\wedge\tau_U})$ is a semimartingale under the measure $\BP_x$ for $m$-a.e. $x\in U$, then $[Y]^c_t=[M^{[\eta_U]}]^c_{t\wedge\tau_U}$, $t\ge 0$,
for any  bounded  $\eta_U\in \mathscr D(\EE)$  such that $\eta_U=v$ q.e. on $U$.
Indeed, by the Doob--Meyer decomposition,
\[
Y_t-Y_0=M^x_t+A^x_t,\quad t\ge 0,\quad \BP_x\text{-a.s.}
\]
for a martingale $M^x$ and a continuous finite variation process $A^x$ ($A^x_0=M^x_0=0$).
It is well known that for every $x\in U$,
\[
\sum_{k=1}^n(Y_{tk/n}-Y_{t(k-1)/n})^2\to [Y]_t
\]
uniformly on compacts in probability $\BP_x$. 
On the other hand,  by \cite[(5.2.14),  Theorem A.3.10]{FOT},
\[
\sum_{k=1}^n(\eta_U(X_{tk/n})-\eta_U(X_{t(k-1)/n}))^2\to [M^{[\eta_U]}]_t 
\]
uniformly on compacts in probability $\BP_m$, where $\BP_m(\cdot)=\int_E\BP_x(\cdot)\,m(dx)$. Since $Y_t=\eta_U(X_t),\, t\le\tau_U$, we conclude that 
\begin{equation}
\label{eq2.6}
[Y]^c_t=[M^{[\eta_U]}]_t^c,\quad t\le \tau_U,\quad \BP_x\mbox{-a.s.}
\end{equation}
 for $m$-a.e. $x\in U$.
\end{remark}

\begin{remark}
Let $u\in\mathscr H(D)$,  $\mu^c_{\langle u\rangle}$ be the measure of Remark \ref{rem2.5} and  $Y_t=u(X_{t\wedge \tau_U})$, $t\ge 0$.	
From the definition of $\mu^c_{\langle u\rangle}$ and Remark \ref{rem2.5} it follows  that for any open $U\subset\subset D$ and $x\in U$, 
\begin{equation}
\label{eq2.2}
\int_Uf(y)G_U(x,y)\,\mu^c_{\langle u\rangle}(dy)=\mathbb E_x \int_0^{\tau_U}f(X_s)\, d[Y]^c_s,\quad f\in\mathcal B^+(U).
\end{equation}
Indeed, since  both sides of (\ref{eq2.2}) are $(P^U_t)$-excessive
and the part $\EE^U$ of the form $\EE$ on $U$ satisfies  the absolute continuity condition, if follows from \cite[Chapter VI, Proposition (1.3)]{BG} that it suffices to show (\ref{eq2.2}) for  $m$-a.e. $x\in U$.
To see this, we first note that 
 by Theorem \ref{th2.3} $u\in\dot{\mathscr D}_{loc}(\EE^D)$. Hence, 
by the very definition of this space, there exists a sequence $(V_n)_{n\ge1}$
of finely open sets such that its union is q.e. equal to $D$
and for every $n\ge 1$ there exists  $v_n\in\mathscr D(\EE^D)$
such that $u=v_n$ q.e.  on $V_n$. Using \eqref{eq2.8} and  \eqref{eq2.6} for $f\in\BB^+(U)$ we get
\begin{align*}
\int_{V_n}f(y)G_U(x,y)\,\mu^c_{\langle u\rangle}(dy)&=
\int_{U}f(y)G_U(x,y)\,(\mu^c_{\langle u\rangle})_{\lfloor V_n}(dy)\\
&=\int_{U}f(y)G_U(x,y)\,(\mu^c_{\langle v_n\rangle})_{\lfloor V_n}(dy)\\
&=\int_{U}f(y)\mathbf1_{V_n}(y)G_U(x,y)\,\mu^c_{\langle v_n\rangle}(dy)\\
&=\BE_x \int_0^{\tau_U}\mathbf1_{V_n}(X_s)f(X_s)\, d[M^{[v_n]}]^c_s\\
&=\BE_x \int_0^{\tau_U}\mathbf1_{V_n}(X_s)f(X_s)\, d[Y]^c_s.
\end{align*}
Letting $n\to \infty$ yields \eqref{eq2.2}.
\end{remark}

We close this subsection with an analytic characterization of harmonic functions obtained in \cite{C}. For $u\in\BB(E)$ let us consider the following conditions: for any open relatively compact open sets $V,U$ such that  $U\subset \subset V \subset\subset D$,
\begin{equation}
\label{eq2.3}
\int_{U \times(E \setminus V)}|u(y)| \,J(d x, d y)<\infty
\end{equation}
and
\begin{equation}
\label{eq2.4}
\mathbf{1}_U(x) \mathbb{E}_x\left[\left(\left(1-\phi_V\right)|u|\right)\left(X_{\tau_U}\right)\right] \in \mathscr D_e(\EE^U)
\end{equation}
for some $\phi_V \in C_c(D) \cap\DDr(\EE)$ such that  $0 \leq \phi_V \leq 1$ and $\phi_V=1$ on $V$. 
Note that both conditions are automatically satisfied when $X$ is a diffusion or $u$ is bounded.

Let $D$ be an open  relatively compact subset of  $E$.
It is known (see \cite{C} or \cite[Lemma 6.7.8]{CF}) that  $\EE$ can be extended to
all pairs $(u,\eta)$ such that  $u\in \mathscr D_{loc}(\EE^D)$ is locally bounded on $D$ and satisfies  \eqref{eq2.3} and $\eta\in C_c(D)\cap\mathscr D(\EE)$  by letting 
\[
\EE(u,\eta):= \EE^{(c)}(u_\eta,\eta)+\int_{E\times E} (u(x)-u(y))(\eta(x)-\eta(y))\,J(dx,dy)+\int_Eu(x)\eta(x)\kappa(dx),
\]
where $u_\eta$ is an arbitrary function in $\mathscr D(\EE^D)$ such 
that $u_{\eta}=u$ on an open set $V$ such that 
$\mbox{supp}[\eta]\subset V\subset\subset D$
(the right-hand side does not depend on the choice of $u_\eta$ due to the strong local property of $\EE^{(c)}$).

\begin{theorem}
Suppose that $u \in \mathscr D_{loc}(\EE^D)$ is locally bounded on $D$,
satisfies  \eqref{eq2.3}, \eqref{eq2.4}  and 
\[
\EE(u, v)=0 \quad \text { for every } v \in C_c(D) \cap \mathscr D(\EE).
\]
Then $u\in\mathscr H(D)$.
\end{theorem}
\begin{proof}
See \cite{C} or \cite[Theorem 6.7.9]{CF}.
\end{proof}

For examples of forms and functions satisfying (\ref{eq2.3}), (\ref{eq2.4}) we refer the reader to \cite{C} and \cite[Section 6.7]{CF}.

\subsection{It\^o--Meyer formula}

In this subsection, $(\Omega,\FF,\BF,\mathbb P)$ is an arbitrary  filtered, complete probability space satisfying the usual hypotheses (see, e.g., \cite[Chapter I]{Pr}. We denote by $\BE$ the expectation with respect to $\mathbb P$. For a c\` adl\`ag process $Y$ and $t>0$ we set $\Delta Y_t=Y_t-Y_{t-}$\,.

Let $Y$ be a c\`adl\`ag  $\BF$-semimartingale. For $a\in\BR$ we denote by $[0,\infty)\ni t\mapsto L^a(Y)_t$ (or simply by $L^a_t$) the local time of $Y$ at $a\in\BR$. Note that we can always choose a jointly measurable (in $a$ and $t$) version of the local time, $L^a_0=0$ and  the process $L^a$ is  continuous and  increasing in $t$ for each $a$  (see \cite[Section  IV.7]{Pr} for details).
For the following result see  \cite[Theorem IV.70]{Pr}.

\begin{proposition}[It\^o--Meyer formula]
Let $Y$ be a c\`adl\`ag $\BF$-semimartingale and $\varphi:\mathbb R\to\mathbb R_+$
be a convex function. Then 
\begin{align}
\label{eq2.1}
\varphi(Y_t)&=\varphi(Y_0)+\int^t_0\varphi'_{-}(Y_{s-})\,dY_s+\frac12\int_{\mathbb R}L^a_t\,\mu_{\varphi}(da)+J^\varphi_t,
\quad t\ge0,
\end{align}
with
\[
J^\varphi_t=\sum_{0<s\le t}\{\varphi(Y_s)-\varphi(Y_{s-})-\varphi'_{-}(Y_{s-})\cdot\Delta Y_s\}.
\]
\end{proposition}

\begin{remark}
By the occupation time formula (see \cite[Corollary IV.1]{Pr}), for every $f\in\mathcal B^+(\mathbb R)$,
\[
\int_{\mathbb R}L^a_t f(a)\,da=\int_0^t f(Y_s)\,d[Y]^c_s,\quad t\ge 0.
\]
Therefore if  (\ref{eq1.2}) is satisfied, then  \eqref{eq2.1} takes the form 
\[
\varphi(Y_t)=\varphi(Y_0)+\int^t_0\varphi'_{-}(Y_{s-})\,dY_s
+\frac12\int_0^tg(X_{s})\,d[Y]^c_s+J^\varphi_t,
\quad t\ge0.
\]
In particular, if  $\varphi(x)=|x|^p$, then
\begin{align}
\label{eq2.5}
|Y_t|^p&=|Y_0|^p+p\int^t_0\mbox{\rm sgn}(Y_{s-})|Y_{s-}|^{p-1}\,dY_s\nonumber\\
&\quad +\frac{p(p-1)}{2}\int^t_0\fch_{\{Y_s\neq0\}}|Y_s|^{p-2}\,d[Y]_s^c+J_t,
\quad t\ge0,
\end{align}
where $[Y]^c$ denotes the continuous part of the quadratic variation $[Y]$
of $Y$ and
\[
J_t=\sum_{0<s\le t}\{|Y_s|^p-|Y_{s-}|^p-p|Y_{s-}|^{p-1}
\mbox{\rm sgn}(Y_{s-})\cdot\Delta Y_s\}.
\]
\end{remark}

\begin{proposition}
\label{prop2.6}
Let $\varphi:\mathbb R\to\mathbb R_+$
be a convex function and $Y$ be a c\`adl\`ag $\BF$-adapted process.
Suppose that $\tau$ is an  $\BF$-stopping time and $(Y_{t\wedge\tau})_{t\ge 0}$ 
is a uniformly integrable  martingale.
Then
\begin{equation}
\label{eq2.7}
\mathbb E\varphi(Y_\tau)=\mathbb E\varphi(Y_0)+\frac12\mathbb E\int_{\mathbb R}L^a_\tau\,\mu_{\varphi}(da)+\mathbb EJ^\varphi_\tau.
\end{equation}
In particular, the  expression on the  left-hand side is finite if and only if that on the right-hand side is.
\end{proposition}
\begin{proof}
Set
\[
M_t=\int^{t\wedge\tau}_0\varphi'_{-}(Y_{s-})\,dY_s, \quad t\ge0.
\]
By \cite[Theorem IV.29]{Pr}, $M$ is a local martingale.
Suppose  that $\mathbb E\varphi(Y_{\tau})<\infty$.
Let $Z_t=\varphi(Y_{\tau\wedge t})$, $t\ge0$. By  \cite[Corollary of Theorem I.18]{Pr}, for every stopping time $\sigma$, 
\[
\BE(Z_\tau|{\FF_\sigma})=\BE(\BE(\varphi(Y_{\tau})|\FF_{\tau})|\FF_{\sigma})
=\BE(\varphi(Y_{\tau})|\FF_{\tau\wedge\sigma}).
\]	
Therefore, by  Jensen's inequality and Doob's Optional Sampling Theorem (see \cite[Theorems I.16, I.19]{Pr}),
\[
\BE(Z_\tau|{\FF_\sigma})=\BE(\varphi(Y_{\tau})|\FF_{\sigma})
\ge \varphi(\BE(Y_{\tau}|\FF_{\tau\wedge\sigma}))=\varphi(Y_{\tau\wedge\sigma})=Z_{\sigma}.
\]
As a result the family $(Z_\sigma)$ over all stopping times $\sigma$ is uniformly integrable.
Let $(\tau_n)$ be a fundamental sequence for the  local martingale $M$. By the definition, for each $n\ge1$,  the stopped process $M^{\tau_n}=M_{\cdot\wedge\tau_n}$ is a uniformly integrable martingale. Therefore from  (\ref{eq2.1}) it follows that 
\[
\BE\varphi(Y_{\tau\wedge\tau_n})=\BE\varphi(Y_0)
+\frac12\mathbb E\int_{\BR}L^a_{\tau\wedge\tau_n}\,\mu_{\varphi}(da)
+\BE J^\varphi_{\tau\wedge\tau_n}.
\]
Letting $n\to \infty$ and applying  the Vitali theorem and the monotone convergence theorem we get (\ref{eq2.7}). 
Now suppose that the right-hand side of (\ref{eq2.7}) is finite.
Then, by Fatou's lemma,
\[
\BE\varphi(Y_{\tau})\le\liminf_{n\rightarrow\infty} \BE\varphi(Y_{\tau\wedge\tau_n})
\le\BE\varphi(Y_0)+\frac12\mathbb E\int_{\BR}L^a_{\tau\wedge\tau_n}\,\mu_{\varphi}(da)+\BE J^\varphi_{\tau}.
\]
Hence $\BE\varphi(Y_{\tau})<\infty$ and (\ref{eq2.7}) holds true.
\end{proof}

\begin{corollary}
\label{cor2.7}
Let $\varphi:\BR\to\BR_+$
be a convex function such that \mbox{\rm(\ref{eq1.2})} is satisfied
and $Y$ be a c\`adl\`ag $\BF$-adapted process.
Suppose that $\tau$ is an  $\BF$-stopping time and $(Y_{t\wedge\tau})_{t\ge 0}$ 
is a uniformly integrable  martingale. Then
\[
\BE\varphi(Y_\tau)=\BE\varphi(Y_0)+\frac12\BE\int_0^\tau g(Y_{s})\,d[Y]^c_s
+\BE J^\varphi_\tau.
\]
\end{corollary}

\section{Hardy--Stein type  identities}
\label{sec3}
We first prove Hardy--Stein identities for convex functions satisfying (\ref{eq1.2}) 
and then in the general case. 


\begin{theorem}
\label{th3.1}
Let $\varphi:\mathbb R\to\mathbb R_+$
be a convex function such that \mbox{\rm(\ref{eq1.2})} is satisfied. 
Assume that $u\in\mathscr H(D)$.
Then \mbox{\rm(\ref{eq1.4})} holds  for every  $x\in U\subset\subset D$.
\end{theorem}
\begin{proof}
By Theorem \ref{th2.3}, $u\in\dot{\DDr}_{loc}(\EE^D)$.
Let
\begin{equation}
\label{eq3.7}
Y_t=u(X_{t\wedge\tau_U}),\quad t\ge0.
\end{equation}
We know that $Y$ is a uniformly integrable martingale under $\mathbb P_x$ for any  $x\in D$. Hence, by  Corollary \ref{cor2.7}, for  $x\in U$  we have
\begin{equation}
\label{eq3.1}
\BE_x\varphi(u(X_{\tau_U}))=\BE\varphi(u(X_0))
+\frac12\BE_x\int^{\tau_U}_0g(u(X_s))\,d[Y]^c_s
+\BE_xJ^\varphi_{\tau_U}.
\end{equation}
We next  show that the right-hand side of (\ref{eq3.1}) has the representation given in (\ref{eq1.4}). We start with the representation of the term $\mathbb E_xJ^\varphi_{\tau_U}$.
First note that since $u$ is quasi-continuous,  $Y_{s-}=u(X_{s-})$ for $s\in [0,\tau_U]$ (see \cite[Theorem 4.2.2]{FOT}).
Consider a L\`evy system  $(N,H)$ for $\BM$ (see, e.g., \cite[Appendix A.3]{FOT} or \cite[A.3.4]{CF}).  By \cite[(5.3.5), Theorem 5.3.1]{FOT}, the jumping measure $J$ and the killing measure $\kappa$ in the Beurling--Deny decomposition of $\EE$ are equal to
\[
J(dx\,dy)=(1/2)N(x,dy)\,\mu_H(dx),\qquad \kappa(dx)=N(x,\{\partial\})\,\mu_H(dx),
\]
where $\mu_H$ is the Revuz measure of the positive continuous additive functional $H$.
By \cite[(A.3.23)]{FOT} (or \cite[(A.3.31)]{CF}), we have
\begin{align}
\label{eq3.3}
\BE^\varphi_xJ_{\tau_U}
&=\BE_x\sum_{0<s\le\tau_U}F_\varphi(u(X_{s-}),u(X_{s}))\nonumber\\
&=\BE_x\int^{\tau_U}_0
\Big(\int_{E_\partial }F_\varphi(u(X_s),u(y))N(X_s,dy)\Big)\,dH_s
=R^U(\Phi\cdot\mu_H)(x),
\end{align}
where
\[
\Phi(z)=\int_{E}F_\varphi(u(z),u(y))N(z,dy)
+F_{\varphi}(u(z),0)N(z,\{\partial\}).
\]
Hence
\begin{align}
\label{eq3.4}
\BE_xJ^\varphi_{\tau_U}&=\int_UG_U(x,z)
\Big(\int_{E_\partial}F_\varphi(u(z),u(y))N(z,dy)\Big)\,\mu_H(dz)\nonumber \\
&=2\int_UG_U(x,z)\Big(\int_{E}F_\varphi(u(z),u(y))J(dz,dy)\Big)
+\int_UG_U(x,z)F_{\varphi}(u(z),0)\,\kappa(dz).
\end{align}
From (\ref{eq3.1})--(\ref{eq3.4}) and (\ref{eq2.2}) we get (\ref{eq1.4}), which proves the theorem.
\end{proof}

We shall see in Section \ref{sec5} that in some interesting situations the measure $\mu_{\langle u\rangle}^c$ can be identified.

We now turn to the general case. In what follows, for a harmonic function $u$ in $D$,  
$[0,\infty)\ni t\mapsto L^a_t$ denotes the local time at $a$ of the martingale $Y$ defined by  (\ref{eq3.7}). By the construction of local time (see, e.g., \cite[Section IV.7]{Pr}) and (\ref{eq3.1}), 
for every $a\in\BR$ the  process $L^a$ is a positive continuous additive functional of $\BM$ in the strict sense.
We denote by $l_a$ the unique smooth Borel measure on $E$ in the Revuz correspondence with 
$L^a$ (see \cite[Theorem 5.1.4]{FOT} or \cite[Theorem 4.1.1]{CF}).

\begin{theorem}
\label{th3.3}
Let $\varphi:\BR\to\BR_+$ be a convex function and  $u\in\mathscr H(D)$. Then \mbox{\rm(\ref{eq1.4})} holds  with the term involving $g$ replaced by
\mbox{\rm(\ref{eq1.7})}.
\end{theorem}
\begin{proof}
By Proposition \ref{prop2.6}, instead of (\ref{eq3.1}) we have
\[
\BE_x\varphi(u(X_{\tau_U}))=\BE_x\varphi(u(X_0))
+ \frac12\BE_x\int_{\BR}L^a_{\tau_U}\,\mu_{\varphi}(da)
+\BE_xJ^\varphi_{\tau_U}
\]
for $x\in U$. On the other hand, by the Revuz correspondence, 
\begin{equation}
\label{eq3.8}
\BE_xL^a_{\tau_U}=R^Ul_a(x)=\int_UG_U(x,z)\,l_a(dz).
\end{equation}
Using the above equalities and (\ref{eq3.4}) we get the desired result.
\end{proof}

\begin{remark}
The measure $l_a$ can be characterized  without recourse to the notion of local time. 	
Namely, $l_a$ is the unique positive smooth measure on $D$ such that  for every open $U\subset\subset D$,
\begin{equation}
\label{eq3.9}
R^Ul_a(x)=\int_{U^c}|u(z)-a|P_U(x,dz)- |u(x)-a|-R^Uj_a(x),\quad x\in U,
\end{equation}
where 
\begin{align*}
j_a(d x)&=2 \int_{E}\{|u(y)-a|-|u(x)-a|-\sgn(u(x)-a)(u(y)-u(x))\} J(dx,dy)\\
&\quad+\int_E\{|a|-|u(x)-a|+\sgn(u(x)-a)u(x)\,\kappa(dx)\}.
\end{align*}
Indeed, (\ref{eq3.9}) follows from (\ref{eq3.8}), definition of local time  (see \cite[p. 216]{Pr}) and \cite[(A.3.33)]{CF}. 
Furthermore, if $\nu$ is a smooth measure on $D$ such that $R^Ul_a=R^U\nu$ for every $U\subset\subset D$, then $l_a=\nu$ on $D$
(this follows easily from the resolvent identity and  \cite[Theorem 5.1.3]{FOT}).
\end{remark}

\section{Conditional Hardy--Stein identities}

\begin{lemma}
\label{lem4.1}
If $h$ is a nonnegative function on $E$ that is harmonic in $D$ and $h>0$ q.e. in $D$, then $1/h\in\dot{\DDr}_{loc}(\EE^D)$.
\end{lemma}
\begin{proof}
Let $V_n=\{h> 1/n\}$, $n\ge1$. Since $h$ is nonnegative in $E$ and harmonic in $D$,  it is  $(P^{V_n}_t)$-excessive. From this and the fact that $(0,\infty)\ni y\mapsto n-1/y$ is concave it follows that  $n-1/h$ is $(P^{V_n}_t)$-excessive for each $n\ge1$.	
Let $U_n$ be a finely open set such that $U_n\subset V_n$
and $\mbox{Cap}_{\EE^{V_n}}(U_n)<\infty$. By \cite[Proposition III.1.5]{MR},  there exists a $(P^{V_n}_t)$-excessive function $\eta_n\in\mathscr D(\EE^{V_n})$ such that $n-1/h=\eta_n$ q.e. on $U_n$. This together with \cite[Lemma 5.2]{KR} and the fact that $\bigcup_{n\ge 1}V_k=D$ shows that $n-1/h\in\dot{\DDr}_{loc}(\EE^D)$. Since  any constant belongs to $\dot{\DDr}_{loc}(\EE^D)$, this proves the lemma.  
\end{proof}

In what follows we denote  by $\mathscr H^{(s)}(D)$  the set of all singular harmonic functions on $D$ that is the set of all $u\in \mathscr H(D)$ such that $u=0$ a.e. in $D^c$. Clearly, in case $\EE$ is local, $\mathscr H^{(s)}(D) =i_D(\mathscr H(D))$, where $i_D(u):=\mathbf1_D u$ (see Remark \ref{rem2.2}).

\begin{theorem}
\label{th4.1}
Let $\varphi:\BR\to\BR_+$
be a convex function such that  \mbox{\rm(\ref{eq1.2})} is satisfied and let 
$U\subset\subset D$. Suppose that  $u, h\in \mathscr H^{(s)}(D)$,  and 
$h$ is strictly positive in $D$.  Then, for every $x\in U$,
\begin{align}
\label{eq4.1}
\int_{U^c}\varphi\Big(\frac{u(z)}{h(z)}\Big)h(z)\, P_U(x,dz)&=h(x)\varphi\Big(\frac{u(x)}{h(x)}\Big)+
\frac12\int_UG_U(x,z)h(z)g\Big(\frac{u(z)}{h(z)}\Big)\, \mu^c_{\langle u/h\rangle}(dz)
\nonumber \\
&+2\int_U G_U(x,z) \int_{E}F_\varphi\Big(\frac{u(z)}{h(z)},\frac{u(y)}{h(y)}\Big)h(y)\,J(dz,dy).
\end{align}
\end{theorem}
\begin{proof}
First note that by Theorem \ref{th2.3},  Lemma \ref{lem4.1}, and 
\cite[Theorem 1.4.2]{FOT}$ u/h\in\dot{\DDr}_{loc}(\EE^D)$. Therefore  $\mu^c_{\langle u/h\rangle}$ is well defined by Remark \ref{rem2.5}.  Let  $U\subset\subset V\subset\subset D$.
We denote by $\tilde\BP_x$  the law of the process $\mathbb M$ killed upon leaving $D$, and 
next we let    $\tilde{\BP}^h_x$ denote the law of Doob's $h$-transform of $\tilde{\mathbb M}$ 
(for a construction of the transformed process see, e.g., \cite[Section 11.3]{CW}).
Then  $u/h$ is harmonic in $D$ with respect to $\tilde{\mathbb P}^h_x$. Indeed,
by \cite[Theorem 11.9]{CW},
\[
\tilde{\mathbb E}^h_x(u/h)(X_{\tau_U})=\frac{\tilde{\mathbb E}_x\left[\mathbf1_{\{\tau_U<\tau_D\}}(u/h)(X_{\tau_U})h(X_{\tau_U})\right]}{h(x)}
=\frac{\mathbb E_xu(X_{\tau_U})}{h(x)}=(u/h)(x).
\]
By Corollary \ref{cor2.7} applied to the transformed process,  for every  $x\in U$ we have
\begin{align*}
	\int_{U^c}\varphi(u/h)(z)\,\tilde P^h_U(x,dz)
	&=\varphi(u/h)(x)+\frac12\tilde\BE_x^h\int_0^{\tau_U}
	g(u/h)(X_s)\,d[(u/h)(X)]^c_s \nonumber\\
	&\quad+\tilde\BE_x^h\sum_{0< s\le \tau_U}F_\varphi((u/h)(X_{s-}),(u/h)(X_s)).
\end{align*}
From the above equality and  the definition of the process $Z$ and Doob's $h$-transform  we get
\begin{align*}
	\int_{U^c}\varphi(u/h)(z) \frac{h(z)}{h(x)}P_U(x,dz)
	&=\varphi(u/h)(x)\\
	&\quad+\frac{1}{2h(x)}\BE_x\int_0^{\tau_U}
	g(u/h)(X_s)h(X_s)\,d[(u/h)(X)]^c_s  \nonumber\\
	&\quad+\frac{1}{h(x)}\BE_x\sum_{0< s\le \tau_U}F_{\varphi}((u/h)(X_{s-}),(u/h)(X_s))h(X_{s})
\end{align*}
(see \cite[Theorem 3.1]{Song}). 
By using L\'evy system for $\mathbb M$, we conclude   that the last term above equals
\begin{align*}
\frac{1}{h(x)}\int_UG_U(x,z)h(z)
\Big(\int_{E_\partial}F_\varphi(\frac{u}{h}(z),\frac{u}{h}(y))N(z,dy)\Big)\,\mu_H(dz).
\end{align*}
To get (\ref{eq4.1}) we now repeat the arguments from the proof of 
Theorem \ref{th3.1} (we  use  the Revuz correspondence between $[(u/h)(X)]^c$ and the measure $\mu^c_{\langle u/h\rangle}$).
\end{proof}

We now turn to the case of general convex function.

\begin{theorem}
Let $u,h$ satisfy the assumptions of Theorem \ref{th4.1}.
Then, for every $x\in U$, \mbox{\rm(\ref{eq4.1})} holds true with the term involving $g$ replaced by 
$\frac12\int_{\BR}\int_UG^h_U(x,z)\,l^h_a(dz)\,\mu_\varphi(da)$,  where $G^h_U, l^h_a$ are defined as $G_U,l_a$ but with the Markov process $\mathbb M$ replaced by its Doob's $h$-transform.
\end{theorem}
\begin{proof}
It suffices to repeat the proof Theorem \ref{th4.1} but now we use  Proposition \ref{prop2.6} instead of Corollary	\ref{cor2.7}.
\end{proof}

\section{Hardy spaces,   conditional Hardy spaces and Littlewood--Paley type estimates}

In this  section,  we assume that $\varphi$ satisfies \eqref{eq1.2}. We stress, however, that
all the presented results   have natural counterparts in the general case.

\begin{definition}[Hardy spaces]
Let $x\in D$.
We denote by  $\mathscr H^\varphi_x (D)$  the harmonic Hardy space on $D$ that consists of functions $u\in\mathscr H(D)$ satisfying 
\[
|u|_{\mathscr H^\varphi_{x}(D)}:=\sup_{U\subset\subset D}\int_{U^c}\varphi(u(z))\, P_U(x,dz)<\infty.
\]
\end{definition}

\begin{definition}[Conditional Hardy spaces]
Let $x\in D$. For  $h\in \mathscr H^{(s)}(D)$ that is strictly positive in $D$ we denote by $\mathscr H^\varphi_{x;h} (D)$  the conditional harmonic Hardy space on $D$  that consists
of functions $u\in\mathscr H^{(s)}(D)$ satisfying 
\[
|u|_{\mathscr H^\varphi_{x;h}(D)}:=\sup_{U\subset\subset D}\int_{U^c}\varphi\Big(\frac{u(z)}{h(z)}\Big)\frac{h(z)}{h(x)}\, P_U(x,dz)<\infty.
\]
\end{definition}

\begin{theorem}
\label{th5.3}
Let the assumptions  of Theorem \ref{th3.1} hold and  $\EE^D$ be transient. Then for  every $x\in D$,
\begin{align}
\label{eq5.8}
&|u|_{\mathscr H^\varphi_x(D)}=\varphi(u(x))+
\frac12\int_DG_D(x,z)g(u(z))\,\mu_{\langle u\rangle}^c(dz)
\nonumber\\
&\qquad+\int_DG_D(x,z)\Big(\int_{E}F_\varphi(u(z),u(y))J(dz,dy)
+\int_EF_{\varphi}(u(z),0)\,k(dz)\Big).
\end{align}
If, in addition, $u(X_{\tau_D})\in L^1(\mathbb P_x)$ and $u(x)=\mathbb E_xu(X_{\tau_D})$,  $x\in D$, then  the left-hand side of \mbox{\rm(\ref{eq5.8})} can be replaced by $\BE_x\varphi(u(X_{\tau_D}))$.
\end{theorem}
\begin{proof}
Follows from Theorem \ref{th3.1}. We argue as in the proof of \cite[Proposition 3.4]{BGPPR2} (see also the proof of \cite[Lemma 6.1.5]{CF} with $\varphi=u(X_{\tau_D})$).
\end{proof}

\begin{theorem}
\label{th5.4}
Let the assumptions  of Theorem \ref{th4.1} hold and  $\EE^D$ be transient. Then for  every $x\in D$,
\begin{align}
\label{eq5.9}
h(x)|u|_{\mathscr H^\varphi_{x;h}(D)}&=\varphi\Big(\frac{u(x)}{h(x)}\Big)+
\frac12\int_DG_D(x,z)h(z)g\Big(\frac{u(z)}{h(z)}\Big)\, \mu^c_{\langle u/h\rangle}(dz)  \nonumber \\
&+2\int_U G_D(x,z) \int_{E}F_\varphi\Big(\frac{u(z)}{h(z)},\frac{u(y)}{h(y)}\Big)h(y)\,J(dz,dy).
\end{align}
\end{theorem}
\begin{proof}
The proof goes as the proof of Theorem \ref{th5.3} but now we use Theorem \ref{th4.1}.
\end{proof}

We now turn to integrable versions of Hardy norms. For a probability measure $\nu$ on $D$ we set
\[
|u|_{\mathscr H^\varphi_\nu(D)}= \int_D|u|_{\mathscr H^\varphi_x(D)}\,\nu(dx),   \qquad|u|_{\mathscr H^\varphi_{\nu;h}(D)}= \int_D|u|_{\mathscr H^\varphi_{x;h}(D)}\,\nu(dx).
\]
In case $\varphi(x)=|x|^p$ we write   $\|u\|_{\mathscr H^p_\nu(D)}$,  $\|u\|_{\mathscr H^p_{\nu;h}(D)}$	($L^p$-Hardy norms). 
The following corollary gives an equivalent description of these quantities.

\begin{corollary}
Let   $\nu$ be a probability measure on $D$ and $h\in\mathscr H^{(s)}(D)$ be strictly positive in $D$.
\begin{enumerate}[\rm(i)]
\item If $u\in\mathscr H(D)$, then 
\begin{equation}
\label{eq5.3}
|u|_{\mathscr H^\varphi_\nu(D)}=\sup_{U\subset\subset D}\int_D\int_{U^c}\varphi(u(z))\, P_U(x,dz)\,\nu(dx).
\end{equation}
\item If $u\in\mathscr H^{(s)}(D)$, then 
\[
|u|_{\mathscr H^\varphi_{\nu;h}(D)}=
\sup_{U\subset\subset D}\int_D\int_{U^c}\varphi\Big(\frac{u(z)}{h(z)}\Big)\frac{h(z)}{h(x)}\, P_U(x,dz)\,\nu(dx).
\]
\end{enumerate}
\end{corollary}
\begin{proof}
By Theorem \ref{th5.3}, the left-hand side of (\ref{eq5.3}) is the integral on $D$  with respect to  $\nu(dx)$
of the right-hand side of (\ref{eq5.8}). But by Theorem \ref{th3.1} and the argument from the proof of Theorem \ref{th5.3},  this integral 
is the supremum over $U\subset\subset D$ of the integrals on $U$ with respect to $\nu(dx)$
of the right-hand side of (\ref{eq1.4}), i.e.  is equal to the right-hand side of (\ref{eq5.3}). This shows (i). The proof of (ii) is similar. We use Theorem \ref{th5.4} instead of Theorem \ref{th5.3} and (\ref{eq4.1}) instead of (\ref{eq1.4}).
\end{proof}

Taking $\varphi(x)=|x|^p$ with $p>1$ and integrating  \eqref{eq5.8}  and \eqref{eq5.9} against $\nu(dx)$ we get the 
following result.

\begin{corollary}
Let $p>1$, $\nu$ be a probability measure on $D$ and $h\in \mathscr H^{(s)}(D)$
be strictly positive in $D$. 
\begin{enumerate}[\rm(i)]
\item For any $u\in\mathscr H(D)$,
\begin{align*}
&\|u\|_{\mathscr H^p_\nu(D)}=\int_D|u(x)|^p\,\nu(dx)+
\frac{p(p-1)}{2}\int_D\int_DG_D(x,z)|u(z)|^{p-2}\,\mu_{\langle u\rangle}^c(dz)\,\nu(dx)\nonumber\\
&+\int_D\int_DG_D(x,z)\Big(\int_{E}2F_p(u(z),u(y))J(dz,dy)
+\int_D\int_EF_p (u(z),0)\,k(dz)\Big)\,\nu(dx).
\end{align*}
\item For any $u\in\mathscr H^{(s)}(D)$,
\begin{align*}
\|u\|_{\mathscr H^p_{\nu;h}(D)} &=\int_D\Big|\frac{u(x)}{h(x)}\Big|^{p}\,\nu(dx)\nonumber\\
&\quad+
\frac{p(p-1)}{2}\int_D\int_DG_D(x,z)\frac{h(z)}{h(x)}
\Big|\frac{u(z)}{h(z)}\Big|^{p-2}\,\mu_{\langle u/h\rangle}^c(dz)\,\nu(dx)
\nonumber\\
&\quad+2\int_D\int_DG_D(x,z)\int_{E}F_p
\Big(\frac{u(z)}{h(z)},\frac{u(y)}{h(y)}\Big)\frac{h(y)}{h(x)}J(dz,dy)\,\nu(dx).
\end{align*}
\end{enumerate}
\end{corollary}

\begin{remark}
By \cite[Lemma 6]{BDL}, for every $p>1$, 
\[
F_p(a, b) \asymp(b-a)^2(|a| \vee|b|)^{p-2},\quad a,b\in\BR.
\]
As a result one  may effectively  estimate the $L^p$-Hardy space norm of a harmonic function.
\end{remark}

We close this section with a generalization of  Littlewood--Paley type estimate
for square functions. We start with introducing
the so called {\em Littlewood--Paley square function}  and its conditional version.
For  $u\in\mathscr H(D)$ we set
\[
\frak q_u(x):=\Big( \int_DG_D(x,y)\,\mu_{\langle u\rangle}(dy)\Big)^{1/2}
=\Big( \int_D\int_0^\infty 
p_D(t,x,y)\,\mu_{\langle u\rangle}(dy)\,dt\Big)^{1/2},\quad x\in D,
\]
and for  $u,h\in\mathscr H^{(s)}(D)$ such that   $h$ is strictly positive in $D$ we set
\[
\frak q^h_u(x):=\Big( \int_DG_D(x,y)\frac{h(y)}{h(x)}\,\mu_{\langle u/h\rangle}(dy)\Big)^{1/2}=\Big( \int_D\int_0^\infty p_D(t,x,y)\frac{h(y)}{h(x)}\,\mu_{\langle u/h\rangle}(dy)\,dt\Big)^{1/2}.
\]

In the proof of the next theorem we will use the fact that for any open $V\subset\subset D$ and $x\in V$,
\begin{align}
\label{eq5.4}
\mathbb E_x[u(X)]_{\tau_V}&=\int_VG_V(x,y)\,\mu^c_{\langle u\rangle}(dy)
+ \int_VG_V(x,y)|u(y)-u(z)|^2J(dy,dz)\nonumber\\
&\quad+\int_{V}G_V(x,y)|u(y)|^2\,\kappa(dy)
=\int_VG_V(x,y)\,\mu_{\langle u\rangle}(dy).
\end{align}
This follows from \cite[(5.3.3), (5.3.9), (5.3.10)]{FOT} and (\ref{eq2.2}) . Similarly, we have
\begin{align}
\label{eq5.5}
h(x)\BE^h_x[u(X)]_{\tau_V}
&=\int_DG_D(x,y)h(y)\,\mu^c_{\langle u/h\rangle}(dy)\\&\quad
+ \int_DG_D(x,y)\left|\frac{u(y)}{h(y)}-\frac{u(z)}{h(z)}\right|^2h(z)J(dy,dz)
\nonumber\\
&=\int_DG_D(x,y)h(y)\,\mu_{\langle u/h\rangle}(dy).
\end{align}

\begin{theorem}[Littlewood--Paley type estimates]
\label{th5.8}
Let  $p\ge 2$. 
\begin{enumerate}[\rm(i)] 
\item There exists $c_p>0$ such that for any   $u\in\mathscr H(D)$,  
\[
|\frak q_u(x)|\le c_p \|u\|_{\mathscr H^{p}_x(D)},\quad x\in D.
\]
\item There exists $c_p>0$ such that for any $u,h\in\mathscr H^{(s)}(D)$,
with strictly positive $h$,
\[
|\frak q^h_u(x)|\le c_p \|u\|_{\mathscr H^{p}_{x;h}(D)},\quad x\in D.
\]
\end{enumerate}
\end{theorem}
\begin{proof}
By Burkholder--Davis--Gundy inequality and  Doob's $L^p$-inequality,  for any $V\subset\subset D$ and $x\in V$ we have
\begin{align}
	\label{eq5.6}
(\mathbb E_x[u(X)]_{\tau_V})^{p/2}\le \BE_x[u(X)]^{p/2}_{\tau_V}&
\le c'_p\BE_x \sup_{t\le\tau_V}|u(X_t)|^p\nonumber\\
&\le c_p\BE_x |u(X_{\tau_V})|^p\le c_p\| u\|^p_{\mathscr H^p_x(D)}.
\end{align}
Therefore letting $V\nearrow D$ and using (\ref{eq5.4}) and Fatou's lemma we obtain
(i). Since (\ref{eq5.6}) holds with 
$\BE_x$ replaced by $\BE^h_x$, using (\ref{eq5.5}) we get  (ii).
\end{proof}

\begin{corollary}
Let $p\ge 2$. Then there exists $c_p>0$ such that 
\[
\|\frak q_u\|_{L^p(\nu)}\le c_p \|u\|_{\mathscr H^{p}_\nu(D)},\,\, u\in\mathscr H(D),\qquad 
\|\frak q^h_u\|_{L^p(\nu)}\le c_p \|u\|_{\mathscr H^{p}_{\nu;h}(D)},\,\, u\in\mathscr H^{(s)}(D)
\]
for any probability  measure $\nu$ on $D$.
\end{corollary}

\section{Examples} 
\label{sec5}

In the examples below, $E=\BR^d$, $m$ is the $d$-dimensional Lebesgue measure, $(\cdot,\cdot)$ is the usual inner product in $\BR^d$,  $D$ is a domain in $\BR^d$ and $U\subset\subset D$.

\subsection{Diffusions with jumps}

\begin{example}
(L\'evy operators). Let $a$ be a symmetric $d$-dimensional matrix such that $\det a>0$  and $\nu$ be a L\'evy measure on $\BR^d$, i.e. a measure on the Borel subsets of $\BR^d$ such that $\nu(\{0\})=0$ and $\int_{\BR^d}(1\wedge|y|^2)\,\nu(dy)<\infty$. 
Consider the form
\[
\EE(u,v)=\frac{1}{2}(a\nabla u,\nabla v)+\int_{\BR^d\times\BR^d\setminus \frak d}
(u(x+y)-u(y))(v(x+y)-v(y))\,\nu(dy)\,dx
\]
with domain $\DDr(\EE)=\{u\in L^2(\BR^d;dx):\EE(u,u)<\infty\}$. Here the  jump term of the form has the jumping measure $J(dx,dy)=\nu(dy-x)\,dx$.
It is well known (see, e.g., \cite[Example 1.4.1]{FOT}) that $\EE$ is a regular Dirichlet form. The process $\BM$ corresponding to $\EE$ is the sum  of a $d$-dimensional  Wiener process with the covariance matrix $a$ and an independent  (purely jump) L\`evy process with the characteristic triplet $(0,\nu,0)$ (see \cite[Chapter 4]{S} for details). Therefore $\EE$ satisfies the absolute continuity condition. Furthermore, by Poincare's inequality,  the part of $\EE$ on $U\subset\subset D$ is transient (see \cite[Example 1.5.1]{FOT}). Therefore Theorems \ref{th3.1} and \ref{th3.3} apply. 

In the special case when $a/2$ is the $d$-dimensional identity matrix $I_d$ and $\nu(dy)=c_{d,\alpha}|y|^{-d-\alpha}$ with $\alpha\in(0,2)$ and suitably chosen constant $c_{d,\alpha}>0$, $\BM$ is the sum of a standard $d$-dimensional Wiener process and a rotationally symmetric  $\alpha$-stable process, and the generator of $\EE$ has the form $\Delta-(-\Delta)^{\alpha/2}$.
\end{example}

\begin{example}
(L\'evy-type operators). Let 
\begin{equation}
\label{eq5.1}
	\EE(u,v)=\EE^{(c)}(u,v)+\EE^{(j)}(u,v),
	\quad u,v\in\DDr(\EE):=H^1(\BR^d)\cap H^{\alpha/2}(\BR^d).
\end{equation}
We assume that  the strongly local part of $\EE$ is given by
\begin{equation}
\label{eq5.2}
\EE^{(c)}(u,v)=\frac12(a\nabla u\,\nabla v):=\frac{1}{2}\sum^d_{i,j=1}\int_{\BR^d}
	a_{ij}(x)\partial_{x_i}u(x)\partial_{x_j}v(x)\,dx,
	\quad u,v\in H^1(\BR^d),
\end{equation}
for some  symmetric $d$-dimensional matrix $a$ such that for some $\lambda\ge1$,
\[
	\lambda^{-1}|\xi|^2\le
	\sum^{d}_{i,j=1}a_{ij}(x)\xi_i\xi_j\le\lambda|\xi|^2,\quad x\in\BR^d,\,\xi=(\xi_1,\dots,\xi_d)\in\BR^d.
\]
As for the jump part, we assume that it has the  form
\begin{equation}
\label{eq6.3}
\EE^{(j)}(u,v)=\int_{\BR^d\times\BR^d\setminus \frak d}
(u(x)-u(y))(v(x)-v(y))J(x,y)\,dx\,dy,\quad u,v\in H^{\alpha/2}(\BR^d),
\end{equation}
with the jumping kernel $J$ defined by
\begin{equation}
\label{eq6.4}
J(x,y)=\frac{c(x,y)}{|x-y|^{d+\alpha}},\quad x,y\in\BR^d,\,x\neq y,
\end{equation}
for some $\alpha\in(0,2)$ and a  measurable function $c:\BR^d\times\BR^d\rightarrow\BR$  such that for any $x,y\in\BR^d$, $c(x,y)=c(y,x)$ and
$C^{-1}\le c(x,y)\le C$
for some $C\ge1$. In the above definitions, $H^1(\BR^d)$ is the usual Sobolev space of order 1 and
\[
H^{\alpha/2}(\BR^d)=\Big\{u\in L^2(\BR^d):\int_{\BR^d\times\BR^d}
\frac{(u(x)-u(y))^2}{|x-y|^{d+\alpha}}\,dx\,dy<\infty\Big\}.
\]
The form $(\EE,\DDr(\EE))$ is regular (see \cite{CK2} or \cite[Example 6.7.16]{CF}).  By Poincare's inequality,  the part of $\EE$ on $U\subset\subset D$ is transient. 
The generator of $\EE$  when restricted to the class $C^{\infty}_c(\BR^d)$ of smooth functions on $\BR^d$ with compact support is given by the formula
\[
Lu(x)=\frac12\sum_{i,j=1}^d\partial_{x_i}(a_{ij}(x)\partial_{x_j}u(x))
+\lim_{\varepsilon\downarrow0}\int_{\{y\in\BR^d:|y-x|>\varepsilon\}}
(u(y)-u(x))J(x,y)\,dx\,dy.
\]
Properties of the associated symmetric Hunt process  $\BM=(X,P_x)$ are thoroughly studied in \cite{CK2}. In particular, it is known that $\BM$ satisfies the absolute continuity condition. 
From Theorem \ref{th3.1} it follows that if  $u$ is a   harmonic function in $D$, then for all $x\in U$ and  $p>1$,
\begin{align}
\label{eq6.5}
E_x|u(X_{\tau_U})|^p&=|u(x)|^p+\frac{p(p-1)}{2}
\int_UG_U(x,z)\fch_{\{u(z)\neq0\}}|u(z)|^{p-2}(a\nabla u,\nabla u)(z)\,dz \nonumber\\
&\quad+\int_UG_U(x,z)\Big(\int_{\BR^d}F_p(u(z),u(y))J(z,y)\,dy\Big)\,dz.
\end{align}
Indeed, an elementary computation shows that for $f\in C^{\infty}_c(D)$ and bounded $\eta\in\DDr(\EE)$ we have
\[
2\EE^{(c)}(\eta,\eta f)-\EE^{(c)}(\eta^2,f)
=\sum^d_{i,j=1}\int_Df(x)a_{ij}(x)\partial_{x_i}\eta(x)\partial_{x_j}\eta(x)\,dx,
\]
so by  the definition of the energy measure,
\begin{equation}
\label{eq5.7}
\mu^c_{\langle u\rangle}(dx)=(a\nabla u,\nabla u)(x)\,dx.
\end{equation}
Therefore (\ref{eq6.5}) follows from Theorem \ref{th3.1} and (\ref{eq1.3}).
\end{example}

\subsection{Diffusions and pure jump processes}

\begin{example}
\label{ex6.3}
(Local operator). Let $\EE=\EE^{(c)}$, where $\EE^{(c)}$ is the form defined by (\ref{eq5.2}) with $a=I_d$. The associated process $\BM$ is a standard $d$-dimensional Wiener process. It is well known that  harmonic functions (relative to $\BM$) in $D$  are smooth in $D$. In \cite[Lemma 19]{BDL} it is proved that 
if $u, h\in\mathscr H^{(s)}(D)$  and $h$ is strictly positive in $D$, then for  any $x\in U$ and $p>1$,
\[
\mathbb E_x\frac{|u(X_{\tau_U})|^p}{h^{p-1}(X_{\tau_U})}
=\frac{|u(x)|^p}{h^{p-1}(x)}+\frac{p(p-1)}{2}
\int_UG_U(x,z)\Big|\frac{u(z)}{h(z)}\Big|^{p-2}\Big|\nabla\frac{u}{h}(z)\Big|^2
h(z)\,dz.
\]
This result follows from  Theorem \ref{th4.1} with $\varphi(x)=|x|^p$ and (\ref{eq5.7}) applied to  $u/h$ in place of  $u$.
\end{example}

\begin{example}
\label{ex6.4}
Let $\EE$ be the form of Example \ref{ex6.3}. By  Theorem \ref{th3.3} applied to the function $\varphi(x)=|x|$, for  
$u\in\mathscr H(D)$ we have
\[
\int_{\partial U}|u(z)|\, P_U(x,dz)
=|u(x)|+\int_DG_D(x,z)\sgn(u(z))\,l_0(dz),
\]
where $l_0$ is the measure in the Revuz correspondence with the local time at 0 of  the process $t\mapsto u(X_{t\wedge\tau_U})$, where $X$ is a standard $d$-dimensional Wiener process. It is worth comparing the above formula with
\[
\int_{\partial U}|u(z)|^p\, P_U(x,dz)
=|u(x)|^p+\frac{p(p-1)}{2}\int_DG_D(x,z)|u(z)|^{p-2}|\nabla u(z)|^2\,dz,
\]
which follows from Theorem \ref{th3.1} applied to $\varphi(x)=|x|^p$ with $p>1$.
\end{example}

\begin{example}
Let $\EE$ be the form of Example \ref{ex6.3} and $u\in\mathscr H^{(s)}(D)$. By  Theorem \ref{th5.8}, for any $p\ge2$ and any $u\in\mathscr H^{(s)}(D)$ and strictly positive
$h\in\mathscr H^{(s)}(D)$ we have
\begin{align*}
\Big(\int_DG_D(x,y)\frac{h(y)}{h(x)}
\Big|\nabla\frac{u}{h}(y)\Big|^2\,dy\Big)^{1/2}\le c_p \|u\|_{\mathscr H^{p}_{x;h}(D)}.
\end{align*}
\end{example}

\begin{example}
\label{ex5.2}
(Integral operators). (i) Let $(\EE^{(j)},H^{\alpha/2}(\BR^d))$ be the form defined by (\ref{eq6.3}) with  $J$ given by (\ref{eq6.4}). It is  regular (see \cite[Example 6.7.14]{CF}). The associated Hunt process $\BM$ is called symmetric $\alpha$-stable-like process on $\BR^d$. 
In \cite{CK1} it is proved (for more general jumping kernels  than (\ref{eq6.4})) that it satisfies the absolute continuity condition. Furthermore, from the bound $c(x,y)\ge C^{-1}$ and \cite[Theorem 1.6.4]{FOT} it follows that the part of $\EE^{(j)}$ on $U$ is transient. 
Therefore from Theorem \ref{th3.1} we get (\ref{eq6.5})  with $a=0$  for any harmonic function $u$ in $D$ and $\varphi$ defined by (\ref{eq1.3}). 
\smallskip\\
(ii) Let $\bar\nu$ be a L\'evy measure of the form $\nu(|x|)\, dx$ for some  non-increasing $\nu:[0,\infty)\rightarrow(0,\infty]$. 
Consider the form  $(\EE^{(j)},\DDr(\EE^{(j)}))$  with kernel  $J(x,y)=\nu(|x-y|)$  and domain
$\DDr(\EE^{(j)})=\{u\in L^2(\BR^d):\EE^{(j)}(u,u)<\infty\}$. Then $\BM$ is a L\'evy process with the characteristic triplet $(0,\bar\nu,0)$.  Assume additionally that $\int_{\BR^d}\nu(|x|)\,dx=\infty$. Then, by \cite[Theorem 27.7]{S}, $\BM$ satisfies the absolute continuity condition. Furthermore, if $U\subset\subset D$, then $U$ is bounded, so the part of $\EE$ on $U$ is transient.
Therefore (\ref{eq6.5}) with $a=0$ holds true for any harmonic function $u$ in $D$  and $\varphi$ defined by (\ref{eq1.3}).
This result strengthens \cite[Proposition 3.3]{BGPPR2} (see also \cite[Lemma 4.12]{BGPPR1} for  $p=2$). 
Identity  (\ref{eq6.5}) with $a=0$ is proved in \cite{BGPPR2} for  Lipschitz $U\subset\subset D$ and under some  additional assumptions on $\nu$. These additional assumptions on $\nu$ imply in particular that any harmonic function on $D$ is of class $C^2(D)$ (see \cite[Theorem 4.9]{BGPPR1}).
\end{example}

\begin{example}
\label{ex6.7}
Consider the form from Example \ref{ex5.2}(ii). Let $u$ be a  harmonic function in $D$. Since $\EE^{(c)}=0$, from Theorem \ref{th3.3}   it follows that for every $U\subset \subset D$ and every $x\in D$,
\begin{align*}
&\int_{U^c}|u(z)|P_U(x,dz)= |u(x)|\\
&\qquad+\int_UG_U(x,z)\int_{\BR^d}\big(|u(y)|-|u(z)|-\sgn(u(z))(u(y)-u(z))\big)\,\nu(|y-z|)\,dy\,dz.
\end{align*}
Consequently, if $0\in D$ and $u(x)=0$, then we have 
\[
\|u\|_{\mathscr H^1_0(D)} =\int_{D}G_{D}(0,z)\int_{\BR^d}\big(|u(y)|-|u(z)|-\sgn(u(z))(u(y)-u(z))\big)\,\nu(|y-z|)\,dy\,dz.
\]
\end{example}

\begin{example}
Let $\EE=\EE^{(j)}$  be the pure jump Dirichlet form with the jump kernel $J(x,y)=c_{d,\alpha}|x-y|^{-d-\alpha}$. It is well known that for a suitably chosen constant  $c_{d,\alpha}>0$ the associated generator is the fractional Laplace operator $-(-\Delta)^{\alpha/2}$ and the associated Hunt process  $\BM$ is  a rotationally symmetric $\alpha$-stable process. Therefore $\BM$  satisfies the absolute continuity condition and if $U\subset\subset D$, then   the part of $\EE$ on $U$ is transient.
Let $u$ be a harmonic function in $D$ and $h$ be a  strictly positive harmonic function in $D$. It is known that $u,h$ are continuous. Applying Theorem \ref{th4.1} to $\varphi(x)=|x|^p$ with $p>1$ we get
\[
E_x\frac{|u(X_{\tau_U})|^p}{h^{p-1}(X_{\tau_U})}
=\frac{|u(x)|^p}{h^{p-1}(x)}+\int_UG_U(x,z)\Big(\int_{\BR^d}h(y)
F_p\Big(\frac{u}{h}(z),\frac{u}{h}(y)\Big)|z-y|^{-d-\alpha}\,dy\Big)\,dz
\]	
for  any $U\subset\subset U$ and $x\in U$. For a different proof see \cite[Theorem 17]{BDL}. 
\end{example}

\section*{Acknowledgments}
T. Klimsiak was supported by the Polish National Science Centre (grant 2022/45/B/ ST1/01095).

\end{document}